\documentclass[draft, 10pt, reqno]{amsproc}
\usepackage{amsthm}
\usepackage{amsfonts}
\usepackage{amsmath}
\usepackage{amssymb}
\usepackage{mathrsfs}
\usepackage{epsfig}
\usepackage{graphicx}
\usepackage{epstopdf}
\usepackage{color}
\usepackage{ifthen}
\usepackage{float}
\usepackage[active]{srcltx}
\usepackage[export]{adjustbox}
\usepackage{tikz}
\usetikzlibrary{arrows.meta,positioning,calc}

\usetikzlibrary{arrows}

\usepackage{marginnote}
\makeatletter
\@namedef{subjclassname@2010}{%
	\textup{2010} Mathematics Subject Classification}
\makeatother

\newtheorem{Theorem}{Theorem}[section]
\newtheorem{Corollary}[Theorem]{Corollary}%[section]
\newtheorem{Lemma}[Theorem]{Lemma}%[section]

\theoremstyle{remark}
\newtheorem{Remark}[Theorem]{Remark}%[section]

\theoremstyle{definition}
\newtheorem{definition}[Theorem]{Definition}%[section]
\newtheorem{example}[Theorem]{Example}

\newcommand{\bq}{\begin{equation}}
	\newcommand{\eq}{\end{equation}}
\newcommand{\beqn}{\begin{eqnarray*}}
	\newcommand{\eeqn}{\end{eqnarray*}}
\newcommand{\beq}{\begin{eqnarray}}
	\newcommand{\eeq}{\end{eqnarray}}

\newcommand{\rar}{\rightarrow}

\newcommand{\bc}{\begin{centre}}
	\newcommand{\ec}{\end{centre}}

\newcommand{\ba}{\begin{array}}
	\newcommand{\ea}{\end{array}}

\newcommand{\eof}{\hfill{{\rule{1.5mm}{1.5mm}}}}

\newcommand{\inp}[2]{\langle{#1},\,{#2} \rangle}

\renewcommand{\Delta}{{\nabla}}

\newcommand*{\childn}[2]{{\mathsf{Chi}}^{\langle#1\rangle}(#2)}

\newcommand*{\Ge}{\geqslant}

\newcommand*{\Le}{\leqslant}

\begin{document}
\title{Ces\`{a}ro Operators on Rooted Directed Trees}

\author[M. Abhinand]{Mankunikuzhiyil Abhinand}
\address{Indian Statistical Institute,
Statistics and Mathematics Unit,
8th Mile, Mysore Road,
Bangalore 560059, India.}
\email{abhinandmkrishnan@gmail.com}

\author[S. Chavan]{Sameer Chavan}
\address{Department of Mathematics and Statistics\\
	Indian Institute of Technology Kanpur, India.}
\email{chavan@iitk.ac.in}

\author[Sophiya S. Dharan]{Sophiya S. Dharan}
\address{Government Polytechnic College, Kalamassery, Kerala, India.}
\email{sophysasi@gmail.com}

\author[T. Prasad]{Thankarajan Prasad}
\address{Department of Mathematics, University of Calicut, Kerala-673635, India.}
\email{prasadvalapil@gmail.com}
\date{}

%\thanks{The research of the %third and fourth authors was supported by the NCN %(National Science Center), decision No. %DEC-2013/11/B/ST1/03613.}

\maketitle

\begin{abstract} In this paper, we introduce and investigate the notion of the Ces\'aro operator $C_{\mathscr T}$ on a rooted directed tree $\mathscr T.$ When $\mathscr T$ is the rooted tree with no branching vertex, then $C_{\mathscr T}$ is unitarily equivalent to the classical Ces\'aro operator $C_{0}$ on the sequence space $\ell^2(\mathbb N).$ We prove that for every narrow rooted directed tree $\mathscr T$, $C_{\mathscr T}$ is bounded, with norm bounded above by twice the width of 
$\mathscr T.$ 
When the tree is not narrow, this boundedness result no longer holds.
Beyond several spectral properties, assuming $\mathscr T$ is leafless and narrow, we show that $C_{\mathscr T}$ is subnormal if and only if $\mathscr T$ is isomorphic to the rooted directed tree without any branching vertex. In particular, this demonstrates that the verbatim analogue of Kriete-Trutt theorem fails in the context of Ces\'aro operators on rooted directed trees. Nonetheless, under the same hypotheses, 
$C_{\mathscr T}$ is always a compact perturbation of a subnormal operator. 
\end{abstract}
	
\section{Introduction}

As mentioned in \cite{R2024}, there has been renewed interest in the classical Ces\'aro operator $C_0$ first explored in \cite{HPL1934} and systematically studied in  \cite{BHS1965} (for its generalizations as of late, see \cite{GP2024, K2021, MR2023, P2008}). The purpose of this paper is to discuss yet another generalization of $C_0$, which is equally motivated by some recent developments in the graph-theoretic operator theory (see \cite{BJJS2017, CGS2023, CPT2017, JJS2012}).  

Let $\mathbb N, \mathbb{Z}_{+}$, $\mathbb{Z}$, $\mathbb{R}$ and $\mathbb{C}$ denote the set of all positive integers, non-negative integers, integers, real numbers and complex numbers, respectively. For a set $X$, $\mathrm{card}(X)$ denotes the cardinality of $X$.
For $\lambda_0 \in \mathbb C$ and $r > 0,$ let $\mathbb D_r(\lambda_0)$ denote the open disc $\{\lambda \in \mathbb C: |\lambda-\lambda_0| < r\}.$
Let $\mathcal{H}$ be a complex Hilbert space and $\mathcal{B}(\mathcal{H})$ denote the $C^*$-algebra of all bounded linear operators on $\mathcal{H}$. For a nonempty subset $M$ of $\mathcal H,$ let $\mbox{span}\,M$ and $\bigvee M$ denote the linear span and closed linear span of $M$ in $\mathcal H,$ respectively. Let $T \in \mathcal{B}(\mathcal{H})$. The Hilbert space adjoint of $T$ is denoted by $T^{\ast}.$ The notation $\ker(T)$ is reserved for the kernel of $T$.
Let $\sigma_{p}(T)$, $\sigma_w(T)$ and $\sigma (T)$ denote the point spectrum, Weyl spectrum and spectrum of $T$, respectively.
%The \textit{Weyl spectrum} of an operator $T\in B(\mathcal{H})$ is the set, 
% $$\sigma_w(T)=\cap_{K\in B_{\infty}(\mathcal{H})}\sigma (T+K),$$ which is the largest part of $\sigma(T)$ that remains unchanged under compact perturbations. That is, $\sigma(T+K)=\sigma(T),~\text{for every} ~K\in B_{\infty}(\mathcal{H}).$ According to a Theorem by \textit{Schechter}, if $T\in B(\mathcal{H})$, then $$\sigma_w(T)=\{ \lambda \in \sigma(T): \lambda ~\text{ is not an isolated eigen value of finite multiplicity}\}.$$
Given $S \in \mathcal{B}(\mathcal{H})$ and $T \in \mathcal{B}(\mathcal{K}),$ we say that $S$ is {\it essentially equivalent} to $T$ if there exists a unitary operator $U : \mathcal H \rar \mathcal K$ and compact operator $K \in \mathcal{B}(\mathcal{H})$ such that $S = U^*TU + K.$
For $A, B \in \mathcal{B}(\mathcal{H}),$ let $[A, B] = AB-BA$ denote the {\it cross commutator} of $A$ and $B.$ 	
An operator $T \in \mathcal{B}(\mathcal{H})$ is called \textit{normal} if $[T^{\ast}, T] = 0,$ {\it essentially normal} if $[T^{\ast}, T]$ is a compact operator, and \textit{hyponormal} if $[T^{\ast}, T] \Ge 0.$ 
An operator $S \in \mathcal{B}(\mathcal{H})$ is
said to be {\it subnormal} if there exist a Hilbert
space $\mathcal K$ containing $\mathcal H$ (an isometric embedding) and a normal operator
$N\in \mathcal{B}(\mathcal{K})$ such that $Nh = Sh$ for every $h \in
\mathcal H.$ Any subnormal operator on a finite dimensional Hilbert space is normal. 
The reader is referred to \cite{C1991} for the basic properties of subnormal and related operators. It is important to note that, in an infinite-dimensional separable Hilbert space, the class of essentially normal operators and the class of operators that are essentially equivalent to subnormal ones do not contain one another (see \cite[Corollaries 1.5.3 $\&$ 5.2.3]{CM2021}).

Let $\ell^2(\mathbb{N})$ denote the complex Hilbert space of all square-summable complex sequences on $\mathbb{N}$. The \textit{Ces\'{a}ro operator} $C_0$ is formally defined by 
\begin{eqnarray} \label{C0}
C_0 (x_1, x_2, x_3, \ldots) := \Big(x_1, \frac{x_1+x_2}{2},  \frac{x_1 + x_2 + x_3}{3}, \ldots \Big).
\end{eqnarray}
It turns out that $C_0$ extends as a bounded linear operator on $\ell^2(\mathbb{N}).$ Indeed, we have the following (the reader is referred to \cite{R2024} for a survey of the various aspects of the Ces\'{a}ro operator):
\begin{Theorem}\label{THM1} \cite[Theorem~9.8.26]{HPL1934}, \cite[Theorems~1-2]{BHS1965}, \cite[Theorems~1 $\&$ 4 and Lemma~1]{KT1971}.
The Ces\'{a}ro operator $C_0$ defines a bounded linear operator on $\ell^2(\mathbb N)$ with operator norm equal to $2.$ Moreover, we have the following statements$:$
\begin{enumerate}
%\item $C_0$ is bounded.
%\item $\lVert 1-C_0\rVert=1$ and $\lVert C_0\rVert =2$,
\item[(a)] $\sigma_p(C_0)=\emptyset,$ $\sigma(C_0)=\overline{\mathbb D_1(1)}$, 
%\item[(b)] $\dim \ker(C^{\ast}_0 - \lambda I)= \begin{cases} 1 & \mathrm{if~}  
%\lambda \in \mathbb D_1(1), \\
%0 & \mathrm{otherwise},
%\end{cases}
%$
\item[(b)] if $x_{\lambda}(0)=1$ and $x_{\lambda}(n)=\prod_{j=1}^n\Big(1-\frac{1}{j\lambda}\Big),$ $n \Ge 1,$ then $$\ker(C^{\ast}_0 - \lambda I)= \begin{cases} \mathrm{span}\{x_{\lambda}\} & \mathrm{if~}  
\lambda \in \mathbb D_1(1), \\
\{0\} & \mathrm{otherwise},
\end{cases}
$$
\item[(c)] $C_0$ is a cyclic subnormal operator.
\end{enumerate}
\end{Theorem}

To introduce the notion of the Ces\'{a}ro operator on a rooted directed tree, we briefly recall some terms from graph-theory (the reader is referred to \cite{JJS2012} for more details). 
A \textit{directed graph} $\mathscr{T} = (V, E)$ is a pair, where $V$ is a nonempty set of \textit{vertices} and $E$ is a subset of $V \times V \setminus \{(v,v) : v \in V\}$ of \textit{edges}. 
%A \textit{circuit} is a collection of finite sequence $\{v_i\}_{i=1}^{n}$ of distinct vertices such that $(v_i, v_{i+1}) \in E$ for $1 \leq i \leq n-1$ and $(v_n, v_1) \in E$, where $n \geq 2$. 
For a subset $W$ of $V$ and $n \in \mathbb{Z}_{+}$, we denote $\text{Chi}(W) = \bigcup\limits_{u \in W} \{v \in V : (u,v) \in E\}$  and $\text{Chi}^{\langle n \rangle}(W) = \text{Chi}(\text{Chi}^{\langle n-1 \rangle} (W)).$ Note that,  $\text{Chi}(v) = \text{Chi}(\{v\})$ and $\text{Chi}^{\langle n \rangle}(v) = \text{Chi}(\text{Chi}^{\langle n-1 \rangle} (\{v\}))$, where $v \in V$. Any element in the set $\text{Chi}(v)$ is called \textit{child of $v$}. If there exists a unique vertex $u \in V$ such that $(u,v) \in E$, then $u$ is called the \textit{parent of $v$} and it is denoted by $\text{par}(v)$. For an integer $k \Ge 1,$ the $k$-fold composition of $\text{par}$ with itself is denoted by $\text{par}^k.$ We will write $\text{par}^k(u)$ only when $u \in V$ belongs to the domain of $\text{par}^k.$ We set $\text{par}^0(v)=v$ for $v \in V.$
A vertex $v \in V$ is called the \textit{root} of $\mathscr{T}$ if there is no $u \in V$ such that $(u,v) \in E$. 
Let $\text{Root}(\mathscr{T})$ denote the colletion of all roots of $\mathscr{T}$. Set $V^{\circ} := V \setminus \text{Root}(\mathscr{T}).$
A \textit{directed tree} is a connected directed graph $\mathscr{T} = (V, E)$ without circuits and 
each vertex $v \in V^\circ$ has a parent. 
It turns out that a rooted directed tree has a unique root (see \cite[p.\,10]{JJS2012}).
A directed tree $\mathscr{T}$ is called \textit{rooted} if it has a root. A directed tree $\mathscr{T} = (V, E)$ is called \textit{locally finite} if $\text{card}(\text{Chi}(u))$ is finite for all $u \in V$. A \textit{leafless} directed tree is a directed tree in which every vertex has atleast one child. A rooted directed tree is {\it narrow} if there exists a positive integer $m$ such that
\beq \label{narrow-c}
\mbox{card}(\text{Chi}^{\langle n \rangle}(\mathsf{root})) \,\leqslant\, m, \quad n \Ge 0.
\eeq
The smallest positive integer $m,$ to be denoted by $\mathrm{width}(\mathscr T),$ satisfying \eqref{narrow-c} will be referred to as the {\it width} of $\mathscr T.$

%(we do not assume that $\mathsf{root} \in V_\mathscr{P}).$  If the root of $\mathscr P$ is $\mathsf{root}$, then we say that $\mathscr P$ is {\it originated from root}. 
 
Let $\mathscr T=(V, E)$ be a rooted directed tree with root $\mathsf{root}$. Then
\beq 
\label{disjoint}
V = \bigsqcup_{n = 0}^{\infty} \childn{n}{\mathsf{root}}~(\mbox{disjoint union})
\eeq
(see \cite[Proposition 2.1.2]{JJS2012}).
For each $u\in V$, let $\mathrm{dep}(u)$ denote the unique non-negative integer 
such that
$u \in \mathsf{Chi}^{\langle \mathrm{dep}(u)\rangle}(\mathsf{root})$. We refer to $\mathrm{dep}(u)$ as the {\it depth} of $u.$ 
Note that for any $u \in V,$ $\text{par}^j(u)$ is defined for $j= 0, \ldots, \text{dep}(u).$
For a rooted directed tree $\mathscr{T} = (V, E)$, let  
$
V_{\prec} = \{u \in V : \text{card}(\text{Chi}(u)) \Ge 2\}
$ 
denote the set of all \textit{branching vertices}. The \textit{branching index} of $\mathscr{T},$ denoted by $k_{\mathscr{T}},$  is defined by
\beq \label{branch-index}
k_{\mathscr{T}} = 
\begin{cases}
	1 + \sup \{\mathrm{dep}(w) : w \in V_{\prec}\} & \text{if } V_{\prec} \text{ is nonempty,}\\
	0 & \text{otherwise.}
\end{cases}
\eeq
A subtree $\mathscr{P} = (V_{\mathscr{P}}, E_{\mathscr{P}})$ of $\mathscr{T} = (V, E)$ is called a \textit{path} if $\mathsf{root} \in V_\mathscr{P}$ and $\text{card}(\text{Chi}_{\mathscr{P}}(v)) = 1$ for all $v \in V_{\mathscr{P}},$ where $\text{Chi}_{\mathscr{P}}(v)$ denotes the child of $v$ with respect to the subtree $\mathscr{P}.$ 
\begin{Remark} \label{rmk-narrow-branching}
Let $\mathscr T=(V, E)$ be a narrow rooted directed tree. Then
\begin{enumerate}
\item Since 
$\mathrm{Chi}^{\langle j \rangle}(w) \subseteq 
\mathrm{Chi}^{\langle j+\mathrm{dep}(w) \rangle}(\mathsf{root})$ for any integer $j \Ge 0,$ $$\sup_{j \Ge 0}\mathrm{card}(\mathrm{Chi}^{\langle j \rangle}(w)) < \infty, \quad w \in V.$$
\item 
$\mathscr T$ is locally finite, and if in addition $\mathscr T$ is leafless, then $\mathscr T$ is of finite branching index. For an example of narrow rooted directed tree, which is not of finite branching index, see Figure~\ref{fig-narrow}.
\end{enumerate}
\end{Remark}

%\begin{figure}
%\begin{tikzpicture}[scale=.6, transform shape]
%\tikzset{vertex/.style = {shape=circle, fill=black, inner sep=1.5pt}}
%\tikzset{edge/.style = {->,> = latex'}}
%
%%[
%%dot/.style = {shape=circle, fill=black, inner sep=1.8pt},
%%>=Stealth, % arrow tip
%%node distance=1cm
%%]
%% vertices
%%\node[vertex] (a) at  (0,0) {$0$};
%\node[vertex] (b) at  (2,0) {$1$};
%\node[vertex] (c) at  (2,2) {$0$};
%\node[vertex] (d) at  (4, 0) {$3$};
%\node[vertex] (e) at  (4, 2) {$2$};
%\node[vertex] (f) at  (6, 0) {$5$};
%\node[vertex] (g) at  (6, 2) {$4$};
%\node[vertex] (h) at  (8, 0) {$7$};
%\node[vertex] (i) at  (8, 2) {$6$};
%\node[vertex] (j) at  (10, 0) {$9$};
%\node[vertex] (k) at  (10, 2) {$8$};
%
%%edges
%\draw[edge] (c) to (b);
%%\draw[edge] (a) to (c);
%\draw[edge] (e) to (d);
%\draw[edge] (c) to (e);
%\draw[edge] (g) to (f);
%\draw[edge] (e) to (g);
%\draw[edge] (g) to (i);
%\draw[edge] (i) to (h);
%\draw[edge] (i) to (k);
%%\draw[edge] (k) to (j);
%\draw[dashed] [->] (k) to (j);
%%\draw[orange,rotate=45,shift={(3 cm,5 cm)}] \draw[edge] (e) to (g);
%\end{tikzpicture}
%\caption{A narrow tree $\mathscr T=(V, E),$ which is not of finite branching index} \label{fig-narrow}
%\end{figure}

\begin{figure}[H]
\begin{center}
\begin{tikzpicture}[>=stealth, node distance=1.8cm]
	
	% main horizontal nodes
	\node (n1) [circle, fill=black, inner sep=1.8pt] {};
	\node (n2) [circle, fill=black, right of=n1, inner sep=1.8pt] {};
	\node (n3) [circle, fill=black, right of=n2, inner sep=1.8pt] {};
	\node (n4) [circle, fill=black, right of=n3, inner sep=1.8pt] {};
	\node (n5) [circle, fill=black, right of=n4, inner sep=1.8pt] {};
	
	% vertical children
	\node (c1) [circle, fill=black, below=1.2cm of n1, inner sep=1.8pt] {};
	\node (c2) [circle, fill=black, below=1.2cm of n2, inner sep=1.8pt] {};
	\node (c3) [circle, fill=black, below=1.2cm of n3, inner sep=1.8pt] {};
	\node (c4) [circle, fill=black, below=1.2cm of n4, inner sep=1.8pt] {};
	\node (c5) [circle, fill=black, below=1.2cm of n5, inner sep=1.8pt] {};
	
	% edges
	\draw[->] (n1) -- (n2);
	\draw[->] (n2) -- (n3);
	\draw[->] (n3) -- (n4);
	\draw[->] (n4) -- (n5);
	
	\draw[->] (n1) -- (c1);
	\draw[->] (n2) -- (c2);
	\draw[->] (n3) -- (c3);
	\draw[->] (n4) -- (c4);
	
	% dashed continuation (right side)
	%\node (n6) [circle, fill=black, right=1.8cm of n5, inner sep=2.5pt] {};
	%\node (c6) [circle, fill=black, below=1.2cm of n6, inner sep=2.5pt] {};
	
	\draw[->, dashed] (n5) -- (c5);
	\draw[->, thick, dashed] (n5) -- ++(1.75,0);
	%\draw[->, dashed] (n6) -- (c6);
\end{tikzpicture}
\end{center}
	\vskip.3cm
\caption{A narrow tree $\mathscr T=(V, E),$ which is not of finite branching index} \label{fig-narrow}
\end{figure}
\vskip.1cm
\noindent
{\it All the directed trees in this paper are rooted and countably infinite.}
\vskip.1cm

%\section{Ces\'aro Operator}
	
For a nonempty set $X,$ let $\ell^2(X)$ denote the Hilbert space of all square summable complex functions on $X$ with the standard  
inner product $$\langle f, g\rangle = \sum_{x \in X}f(x)\overline{g(x)}, \quad f, g \in \ell^2(X).$$
For each $x \in X$, consider the function $e_x : X \to \mathbb{C}$ defined by
$$e_x(y):=
\begin{cases}
1, \quad \text{if } y=x, \\
0, \quad \text{otherwise}.
\end{cases}$$
Note that the set $\{e_x : x \in X\}$ forms an orthonormal basis for $\ell^2(X)$. By the {\it support} of $f \in \ell^2(X),$ we understand the subset $\{x \in X : f(x) \neq 0\}$ of $X.$
\begin{definition}
%\label{def-CO}
Let $\mathscr{T} = (V, E)$ be a directed tree.
We define the linear operator $C_\mathscr{T}$ in $\ell^2 (V)$ by
\begin{eqnarray}\label{EQ3}
\mathcal D(C_\mathscr{T}) &:=& \{f \in \ell^2(V) : C_\mathscr{T}f \in \ell^2(V)\}, \notag\\
C_\mathscr{T}f &:=& \Lambda_{\mathscr T}f, \quad f \in \mathcal D(C_\mathscr{T}), 
\end{eqnarray} 
where $\Lambda_\mathscr T$ is the mapping on the functions $f : V \rar \mathbb C$ given by
\beqn
(\Lambda_{\mathscr T}f)(v)
=\frac{1}{\mathrm{dep}(v)+1}\sum\limits_{j=0}^{\mathrm{dep}(v)} f(\text{par}^j(v)), \quad v \in V.
\eeqn
The operator $C_{\mathscr T}$ will be called the {\it Ces\'aro operator} on the directed tree $\mathscr T.$
	\end{definition}

%\begin{Proposition} Let $\mathscr{T} = (V, E)$ be a directed tree. For $w \in V,$ then the following statements are equivalent$:$
%\begin{enumerate}
%%\item $C_\mathscr{T}$ is a bounded linear operator on $\ell^2(V),$
%\item[(i)] $e_w \in \mathcal D(C_\mathscr{T}),$ 
%\item[(ii)] $\sum_{j=0}^{\infty}\frac{\mathrm{card}(\mathrm{Chi}^{\langle j \rangle}(w))}{(\mathrm{dep}(w)+j+1)^2} < \infty.$
%%$\sup_{n \Ge 0}\mathrm{card}(\mathrm{Chi}^{\langle n \rangle}(\mathsf{root})) < \infty.$
%\end{enumerate}
%\end{Proposition}
%\begin{proof}
%%\begin{Remark} 
%Let $w \in V.$
%%It is easy to see that 
%%\beq \label{criterion-ew} 
%%\mbox{$e_w \in \mathcal D(C_\mathscr{T})$ if and only if}~\sum_{j=0}^{\infty}\frac{\mbox{card}(\mathrm{Chi}^{\langle j \rangle}(w))}{(\mathrm{dep}(w)+j+1)^2} < \infty.
%%\eeq
%Indeed, if $e_w \in \mathcal D(C_\mathscr{T}),$ then {\color{blue} 
%\beq \label{action-C}
%C_\mathscr{T}e_w &=& \sum_{v \in V}\frac{1}{\mathrm{dep}(v)+1}\sum\limits_{j=0}^{\mathrm{dep}(v)} e_w(\text{par}^j(v)) e_v \notag \\
%&=& \sum_{j=0}^{\infty}\frac{1}{\mathrm{dep}(w)+j+1}\sum_{v \in \mathrm{Chi}^{\langle j \rangle}(w)}e_v, \quad w \in V.
%\eeq
%The equivalence now follows from \eqref{disjoint}.
%%{\color{red}As a consequence, $C_\mathscr{T}$ is unbounded if $\mathscr T$ is not locally finite.}
%}
%\end{proof}

\begin{Remark} 
%\label{rmk-def}
Since the convergence in $\ell^2(V)$ yields the pointwise convergence, $C_\mathscr{T}$ is a closed linear operator. 
If $e_w \in \mathcal D(C_\mathscr{T}),$ then  
\beq \label{action-C}
C_\mathscr{T}e_w &=& \sum_{v \in V}\frac{1}{\mathrm{dep}(v)+1}\sum\limits_{j=0}^{\mathrm{dep}(v)} e_w(\text{par}^j(v)) e_v \notag \\
&=& \sum_{j=0}^{\infty}\frac{1}{\mathrm{dep}(w)+j+1}\sum_{v \in \mathrm{Chi}^{\langle j \rangle}(w)}e_v, \quad w \in V.
\eeq
Thus, for $w \in V,$ 
$e_w \in \mathcal D(C_\mathscr{T})$ if and only if  \beq \label{bdd-limit0}
\sum_{j=0}^{\infty}\frac{\mathrm{card}(\mathrm{Chi}^{\langle j \rangle}(w))}{(\mathrm{dep}(w)+j+1)^2} < \infty.
\eeq
%If
%$e_w \in \mathcal D(C_\mathscr{T}),$ then 
%$\sup_{j \Ge 0}\frac{\mathrm{card}(\mathrm{Chi}^{\langle j \rangle}(\mathsf{root}))}{(\mathrm{dep}(w)+j+1)^2} < \infty.$
Thus, if $\mathscr T$ is a narrow directed tree, then by Remark~\ref{rmk-narrow-branching}(1), 
$e_w \in \mathcal D(C_\mathscr{T}).$ 

Assume that $e_w \in \mathcal D(C_\mathscr{T})$ for every $w \in V.$ Then $\mathcal D(C_\mathscr{T})$ 
contains the linear span of $\{e_v : v \in V\},$ and hence  $C_\mathscr{T}$ is a densely defined linear operator. Since
\beqn
\inp{C_\mathscr{T}f}{e_w} &=& \frac{1}{\mathrm{dep}(w)+1}\sum\limits_{j=0}^{\mathrm{dep}(w)} f(\text{par}^j(w))
\\
&=& \Big\langle{f},\,{\frac{1}{\mathrm{dep}(w)+1}\sum_{j=0}^{\mathrm{dep}(w)}
e_{\mathrm{par}^j(w)}}\Big\rangle, \quad f \in \mathcal D(C_\mathscr{T}),
\eeqn
we may conclude that $e_w \in \mathcal D(C^{*}_\mathscr{T})$ and 
\beq \label{action-C*}
C^*_\mathscr{T}e_w = \frac{1}{\mathrm{dep}(w)+1}\sum_{j=0}^{\mathrm{dep}(w)}
e_{\mathrm{par}^j(w)}.
\eeq
This shows that $C^*_\mathscr{T}e_{\mathsf{root}}=e_{\mathsf{root}},$ and consequently, $1 \in \sigma_p(C^*_\mathscr{T}).$
\end{Remark}

If $\mathscr T$ is the rooted tree with no branching vertex, then $C_{\mathscr T}$ is unitarily equivalent to the classical Ces\'aro operator $C_{0}$ on the sequence space $\ell^2(\mathbb N).$ 
\begin{example}
\label{classical-example} 
Consider the directed tree $\mathscr T_1$ with the set of vertices 
$V:=\mathbb{N}$ and $\mathsf{root}=1$. We further require that $\mathsf{Chi}(n)=\{n+1\}$ for
all 
$n \in \mathbb N$. Note that $\mathrm{Chi}^{\langle j \rangle}(n)=\{n+j\}$ for $n, j \in \mathbb N.$ 
Since $\mathrm{dep}(n)=n-1$ and $\text{par}^j(n)=n-j$ for $n \in \mathbb N$ and $0 \Le j \Le n,$ 
it follows that for any $f \in \mathcal D(C_{\mathscr{T}_1}),$ 
\beqn
(C_{\mathscr{T}_1}f)(n) =  
\frac{1}{\mathrm{dep}(n)+1}\sum\limits_{j=0}^{\mathrm{dep}(n)} f(\text{par}^j(n)) = \frac{1}{n}\sum\limits_{j=1}^{n} f(j).
\eeqn  
It turns out that $C_{\mathscr T_1}$ extends as a bounded linear operator on $\ell^2(\mathbb N).$ Indeed, $C_{\mathscr T_1}$ is unitarily equivalent to the Ces\'aro operator $C_0$ (see \eqref{C0}). 
\eof
\end{example}

As in the classical case (see Example~\ref{classical-example}), $C_\mathscr{T}$ extends as a bounded linear operator on $\ell^2(V)$ provided $\mathscr{T} = (V, E)$ is a narrow rooted directed tree. 
Moreover, the following main result of this paper provides a variant of Theorem~\ref{THM1}.
\begin{Theorem} \label{main-thm}
Let $\mathscr{T} = (V, E)$ be a narrow rooted directed tree.
Then the Ces\'{a}ro operator $C_\mathscr{T}$ defines a bounded linear operator on $\ell^2(V)$ satisfying
\beq \label{norm-estimate} 2 \Le \|C_\mathscr{T}\| \Le 2 \sqrt{\mathrm{width}(\mathscr T)}.
\eeq
Moreover, we have
\begin{enumerate}
\item[(a)] $\sigma_p(C_\mathscr{T})=\emptyset$ if and only if $\mathscr T$ is leafless, 
\item[(b)] any $\lambda \in \mathbb D_1(1)$ is an eigenvalue of $C^{\ast}_\mathscr T,$ and if $\mathscr T$ has at least two paths, then $\dim \ker(C^{\ast}_\mathscr T - \lambda I) \Ge 2$ for every
$\lambda \in \mathbb D_1(1) \backslash \{1\}.$ 
\end{enumerate}
Furthermore, if $\mathscr T$ is a leafless rooted directed tree, then
\begin{enumerate}
\item[(c)] $\sigma(C_\mathscr T)=\overline{\mathbb D_1(1)},$ 
\item[(d)] $C_\mathscr{T}$ is subnormal if and only if $\mathscr T$ is isomorphic to the directed tree $\mathscr T_1$,
\item[(e)] $C_\mathscr{T}$ is essentially equivalent to a subnormal operator.
\end{enumerate}
\end{Theorem} 

In general, the norm of the Ces\'aro operator $C_{\mathscr T}$ could be bigger than $2$ (see Example~\ref{k-child-ex}). Also, the assumption that $\mathscr T$ is narrow can not be dropped (see 
Example~\ref{Chi-n-order-1}). 
Parts (a) and (c) are consistent with the classical case (see Theorem~\ref{THM1}(a)). 
Part (b) shows that when $\mathscr T$ has a branching vertex, the eigenvalues of $C^*_{\mathscr T},$ except $1,$ are no longer simple. Also, as shown in part (d), except for the directed tree $\mathscr T$ without branching vertices,  subnormality of $C_{\mathscr T}$ fails rather suprisingly. Nevertheless, the part (e) ensures essential subnormality of $C_{\mathscr T}$. 
    
%Our proof of Theorem~\ref{main-thm} is fairly long and it relies particularly on Theorem~\ref{THM1}. This proof,  consisting of several facts (see Lemmata~\ref{bounded-narrow}-\ref{subnormality}), will be presented in Section~\ref{2}. In Section 3, we conclude the paper with two examples and some questions. 

Our proof of Theorem~\ref{main-thm} is fairly long and hinges crucially on Theorem~\ref{THM1}. It comprises several intermediate results (see Lemmata \ref{bounded-narrow}–\ref{subnormality}) and will be given in Section~\ref{2}. In Section 3, we present two examples and discuss several natural problems that emerge along the way.

\section{Boundedness, spectral properties and subnormality} \label{2}

%Let $\mathcal{T} = (V, E)$ be a rooted locally finite directed tree. By $G_n$, we denote the $n$th generation of $\mathcal T,$ that is, the set of all vertices $v \in V$ with $\mathrm{dep}(v) = n.$ Thus 
%\beq \label{generation}
%V = \sqcup_{n=0}^{\infty} G_n~(\mathrm{disjoint~sum}). 
%\eeq

We begin with the boundedness of the Ces\'aro operator in case of narrow directed trees. 
\begin{Lemma} \label{bounded-narrow}
If $\mathscr{T} = (V, E)$ is a narrow rooted directed tree, then $C_\mathscr{T}$ is a bounded linear operator satisfying $\|C_\mathscr{T}\| \Le 2\sqrt{\mathrm{width}(\mathscr T)}.$
\end{Lemma}
\begin{proof} Assume that $\mathscr{T}$ is a narrow rooted directed tree. We first note that
%for any $f \in \ell^2(V),$
\beq \label{gen-decom}
\|f\|^2 = \sum_{n =0 }^\infty \|f \cdot \chi_{\childn{n}{\mathsf{root}}}\|^2, \quad f \in \ell^2(V).
\eeq
This is a consequence of \eqref{disjoint}.
Assume now that $\mathscr{T}$ is a narrow rooted directed tree.
For an integer $n,$ let $l_n$ denote the cardinality of $\childn{n}{\mathsf{root}}.$
Since $\mathscr{T}$ is locally finite, $l_n <\infty$ for all integers $n \Ge 0.$  
Let $f \in \ell^2(V),$ and note that  
\beqn
(C_{\mathcal{T}} f) \cdot \chi_{\childn{n}{\mathsf{root}}}&=& \sum_{v \in \childn{n}{\mathsf{root}}}\Big(\frac{1}{\mathrm{dep}(v)+1}\sum\limits_{j=0}^{\mathrm{dep}(v)} f(\text{par}^j(v))\Big)e_v \\
&=& \frac{1}{n+1}\sum_{v \in \childn{n}{\mathsf{root}}}\Big(\sum\limits_{j=0}^{n} f(\text{par}^j(v))\Big)e_v \\
&=& \frac{1}{n+1}\sum\limits_{j=0}^{n} \Big(\sum_{v \in \childn{n}{\mathsf{root}}}f(\text{par}^j(v))e_v\Big).
\eeqn
This combined with $|f(w)| \Le \|f \cdot \chi_{\childn{n}{\mathsf{root}}}\|,$ $w \in \childn{n}{\mathsf{root}},$ implies that
\beqn
\|(C_{\mathcal{T}} f)\cdot \chi_{\childn{n}{\mathsf{root}}}\| & \Le & \frac{1}{n+1}\sum\limits_{j=0}^{n} \Big \|\sum_{v \in \childn{n}{\mathsf{root}}}f(\text{par}^j(v))e_v\Big\| \\
& = & \frac{1}{n+1}\sum\limits_{j=0}^{n} \Big(\sum_{v \in \childn{n}{\mathsf{root}}}|f(\text{par}^j(v))|^2\Big)^{1/2} \\
& \Le & \frac{\sqrt{l_n}}{n+1}\sum\limits_{j=0}^{n} \|f \cdot \chi_{\childn{n-j}{\mathsf{root}}}\|.
%& \Le & l_n \sum\limits_{j=0}^{n} \|f \cdot \chi_{G_{n-j}}\|^2
\eeqn
Since $\mathcal{T}$ is a narrow tree, this yields
\beqn
\|C_{\mathcal{T}} f\|^2 &\overset{\eqref{gen-decom}}=&
\sum_{n=0}^{\infty} \|(C_{\mathcal{T}} f)\cdot \chi_{\childn{n}{\mathsf{root}}}\|^2 \\
& \Le & \Big(\sup_{k \Ge 0}\, l_k\Big) \sum_{n=0}^{\infty}\Big(\frac{1}{n+1}\sum\limits_{j=0}^{n} \|f \cdot \chi_{\childn{j}{\mathsf{root}}}\|\Big)^2 \\
&\overset{(\ast)}\Le & 4\,\mathrm{width}(\mathscr T) \sum_{n=0}^{\infty}\|f \cdot \chi_{\childn{n}{\mathsf{root}}}\|^2 \\
&\overset{\eqref{gen-decom}}=& 4\,\mathrm{width}(\mathscr T)\|f\|^2,
\eeqn
where the inequality in ($\ast$) follows from Hardy's inequality (see \cite{H1920}): {\it If $\{a_n\}_{n \in \mathbb{Z}_+}$ is a sequence of non-negative real numbers, then 
\beqn \sum\limits_{n=0}^{\infty} \Big(\frac{1}{n+1}\sum\limits_{j=0}^{n} a_j\Big)^2 \Le 4 \sum\limits_{n=0}^{\infty} a_n^2.
\eeqn
}
Since $f \in \ell^2(V)$ is arbitrary, this completes the proof. 
\end{proof}

We show below that $C_\mathscr T$ has no eigenvalues unless $\mathscr T$ has a leaf. 
\begin{Lemma} \label{point-sp}
%Let $\mathscr{T} = (V, E)$ be a rooted directed tree. 
Let $\mathscr{T} = (V, E)$ be a locally finite rooted directed tree. Then $\sigma_p(C_\mathscr{T})=\emptyset$ if and only if $\mathscr T$ is leafless.
\end{Lemma}
\begin{proof} Assume that $\mathscr T$ is leafless. Suppose for some $f \in \ell^2(V)$ and $\lambda \in \mathbb C,$ $C_\mathscr{T}f = \lambda f.$ Let $v \in V$ be a vertex of smallest depth for which $f(v) \neq 0.$ It follows from \eqref{EQ3} that
\beqn
\lambda f(v) = \frac{1}{\mathrm{dep}(v)+1}\sum\limits_{j=0}^{\mathrm{dep}(v)} f(\text{par}^j v) =  \frac{1}{\mathrm{dep}(v)+1}\,f(v).
\eeqn
As a consequence, $\lambda = \frac{1}{\mathrm{dep}(v)+1}.$ Let $v_1 \in \mathrm{Chi}(v).$ 
Once again by \eqref{EQ3}, 
\beqn
\lambda f(v_1) = \frac{1}{\mathrm{dep}(v_1)+1}\sum\limits_{j=0}^{\mathrm{dep}(v_1)} f(\text{par}^j(v_1)) =  \frac{1}{\mathrm{dep}(v_1)+1}(f(v_1)+f(v)).
\eeqn
However, since $\lambda = \frac{1}{\mathrm{dep}(v)+1},$ we obtain 
\beqn
\frac{1}{\mathrm{dep}(v)+1}\, f(v_1) =  \frac{1}{\mathrm{dep}(v)+2}(f(v_1)+f(v)).
\eeqn
This shows that $f(v_1) \neq 0$ and $f(v_1)=f(v)(\mathrm{dep}(v)+1).$ 
In view of
\beqn
\frac{k+n+1}{k+1}{k+n \choose n}={k+n+1 \choose n}, \quad k \Ge 0, n \Ge 1,
\eeqn
a mathematical induction on $n$ yields a sequence $\{v_n\}_{n \Ge 1}$ such that $v_n \in \mathrm{Chi}^{\langle n \rangle}(v)$ and $f(v_n)=f(v){\mathrm{dep}(v)+n \choose n}.$ However, since $f \in \ell^2(V),$ we have $\sum_{n=1}^{\infty}|f(v_n)|^2 < \infty.$ This is not possible unless $f(v)=0,$ which is contrary to the choice of $v.$ Finally, if $\mathscr T$ has a leaf, say, $v$, then $C_{\mathscr T}e_v=\frac{1}{\mathrm{dep}(v)+1}e_v,$ that is, $\sigma_p(C_\mathscr{T})\neq \emptyset.$   
\end{proof}

Let us analyze the eigenspectrum of $C^*_\mathscr T.$ 
\begin{Lemma} \label{point-sp-*}
%Let $\mathscr{T} = (V, E)$ be a rooted directed tree. 
Let $\mathscr{T} = (V, E)$ be a locally finite rooted directed tree. Then 
$\mathbb D_1(1) \subseteq \sigma_p(C^*_\mathscr T).$
Moreover, if $\mathscr T$ has two paths, then every $\lambda \in \mathbb D_1(1) \backslash \{1\}$ is not a simple eigenvalue. In particular, $\|C_\mathscr{T}\| \Ge 2.$
\end{Lemma}
\begin{proof}
Let $\mathscr P=(V_{\mathscr P}, E_{\mathscr P})$ be any path in $\mathscr T,$ where $V_{\mathscr P}=\{v_n\}_{n \in \mathbb N}$ be the vertex set of $\mathscr P$ such that $v_0=\mathsf{root}$ and $\mbox{Chi}_\mathscr P(v_n)=\{v_{n+1}\},$ $n \in \mathbb N.$ It follows from \eqref{action-C*} that $C^*_{\mathscr T}(\ell^2(V_{\mathscr P})) \subseteq \ell^2(V_{\mathscr P}).$ 
Define the unitary transformation $U : \ell^2(V_{\mathscr P}) \rar \ell^2(\mathbb N)$ by setting 
\beq \label{unitary-eigen-*}
Ue_{v_n}=e_n, \quad n \in \mathbb N. 
\eeq
Note that by \eqref{action-C*}, for $n \in \mathbb N,$
\beqn
UC^*_{\mathscr T}e_{v_n}=\frac{1}{\mathrm{dep}(v_n)+1}\sum_{j=0}^{\mathrm{dep}(v_n)}
Ue_{\mathrm{par}^j(v_n)} = \frac{1}{n+1}\sum_{j=0}^{n}
e_{j}=C^*_0e_n=C^*_0Ue_{v_n}.
\eeqn
It follows that
\beq \label{uni-equi}
U{C^*_{\mathscr T}}|_{\ell^2(V_{\mathscr P})}=C^*_0U.
\eeq
%Thus if
%$T$ denotes the restriction of $C^*_{\mathscr T}$ to $\ell^2(V_{\mathscr P}),$ then $UT=C^*_0U.$ 
Since every $\lambda \in \mathbb D_1(1)$ is an eigenvalue of $C^*_0$ (see Theorem~\ref{THM1}(b)), $\lambda$ is also an eigenvalue of ${C^*_{\mathscr T}}|_{\ell^2(V_{\mathscr P})}.$ The inclusion $\sigma_p({C^*_{\mathscr T}}|_{\ell^2(V_{\mathscr P})}) \subseteq \sigma_p(C^*_{\mathscr T})$ now yields the first part. 

Assume that $\mathscr T$ has two paths and let $\lambda \in \mathbb D_1(1)\backslash \{1\}.$ 
Note that the eigenvector of $C^*_0$ corresponding to the eigenvalue $\lambda$ does not have finite support (see Theorem~\ref{THM1}(b)). By \eqref{unitary-eigen-*} and \eqref{uni-equi}, the eigenvector of $C^*_{\mathscr T}$ corresponding to the eigenvalue $\lambda$ does not have finite support. Also, any two paths in $\mathscr T$ can have at most finitely many common vertices (since if $v$ belongs to the intersection of any two paths, then so does $\text{par}^j(v)$ for any $0 \Le j \Le \mathrm{dep}(v)$). Combining these facts, we may conclude that there are at least two linearly independent eigenvectors corresponding to $\lambda.$
This together with the fact that operator norm is bigger than or equal to the spectral radius completes now the proof.   
\end{proof}

Let $\mathscr{T} = (V, E)$ be a locally finite rooted directed tree with root $\mathsf{root}.$ Assume that $\mathscr T$ is of finite branching index. 
Recall from \cite[Proof of Lemma~5.3]{CT2016} that 
$V$ decomposes as  
\begin{eqnarray}\label{EQ2}
V = \Big(\bigsqcup_{j=0}^{k_{\mathscr{T}}-1} \mathrm{Chi}^{\langle j \rangle} (\mathsf{root})\Big) \bigsqcup \Big(\bigsqcup_{i=1}^d\{v_{i,n}: n \Ge 0\}\Big),
\end{eqnarray}
where $d := \mathrm{card}(\mathrm{Chi}^{\langle k_{\mathscr{T}} \rangle}(\mathsf{root})) \in \mathbb N$,  
%$\mathrm{card}(\mathrm{Chi}(v)) = 1$ for all $v \in \mathrm{Chi}^{k_{\mathscr{T}}}(\mathsf{root})$,  
$\mathrm{Chi}^{\langle k_{\mathscr{T}} \rangle}(\mathsf{root}) = \{v_{i,0}: i =1,2,\ldots,d\}$ and $\mathrm{Chi}(v_{i,n}) = \{v_{i,n+1}\}$ for all integers $n \Ge 0$, $i = 1,2, \ldots, d.$

We now use \eqref{EQ2} to see that the Ces\'aro operator $C_\mathscr{T}$ is a compact perturbation of a finite direct sum of the classical Ces\'aro operator $C_0$ provided the rooted directed tree $\mathscr T$ is locally finite and of finite branching index.
\begin{Lemma} \label{lem-decom}
Let $\mathscr{T} = (V, E)$ be a locally finite rooted directed tree with root $\mathsf{root}.$ 
Assume that $\mathscr T$ is of finite branching index $k_\mathscr T$ and let $d := \mathrm{card}(\mathrm{Chi}^{\langle k_{\mathscr{T}} \rangle}(\mathsf{root})) < \infty.$
%Let $W_{-1}$ and $e_{v_{i,n}}$ are same as in $(\ref{EQ1})$ and $(\ref{EQ2})$. 
%Let $\mathcal{M} = \bigvee \{e_v:v \in W_{-1}\}$ and $H_i = \bigvee \{e_{v_{i,n}} : n\geq 0\}$ for $i = 1,2,\ldots, d$. 
Then $\ell^2(V)$ decomposes as  
\begin{eqnarray}\label{EQ6}
\ell^2(V) =\mathcal{M} \oplus \mathcal H_1\oplus  \cdots \oplus \mathcal H_d,
\end{eqnarray}
where $\mathcal{M} = \bigvee \{e_v:v \in \sqcup_{j=0}^{k_{\mathscr{T}}-1} \mathrm{Chi}^{\langle j \rangle} (\mathsf{root})\}$ is a finite dimensional subspace of $\ell^2(V)$ and $\mathcal H_i = \bigvee \{e_{v_{i,n}} : n \Ge 0\},$ $i = 1, \ldots, d$ $($see \eqref{EQ2}$).$ Moreover, 
%the Ces\'aro operator $C_\mathscr{T}$ extends as a bounded linear operator on $\ell^2(V)$ and 
with respect to the decomposition \eqref{EQ6}, $C_\mathscr{T}$ decomposes as
\begin{eqnarray}\label{EQ7}
C_\mathscr{T} = 
\begin{bmatrix}
T & 0 & \cdots & 0\\
A_1 & B_1 & \cdots & 0\\
\vdots & \vdots &  \ddots &\vdots\\
A_d & 0 & \cdots &B_d
\end{bmatrix}, 
\end{eqnarray}
where $B_i : \mathcal H_i \rar \mathcal H_i$ is essentially equivalent to $C_{0},$ $T : \mathcal{M} \to \mathcal{M}$ and $A_i : \mathcal{M} \to  \mathcal H_i $ are bounded linear finite rank operators. 
%If, in addition, $\mathscr T$ is locally finite, then 
%Moreover, $T$ and $A_i$ are finite rank operators. 
\end{Lemma}
\begin{proof} 
The decomposition \eqref{EQ6} is clear from \eqref{EQ2}. Since $\mathrm{Chi}(v_{i,n}) = \{v_{i,n+1}\}$ for all $n \in \mathbb N$, by \eqref{action-C}, 
\beqn
C_{\mathscr T} (\mathrm{span}\,\{e_{v_{i,n}} : n \Ge 0\}) \subseteq \mathrm{span}\,\{e_{v_{i,n}} : n \Ge 0\}, \quad i = 1,2, \ldots, d.
\eeqn 
We claim that the restriction $B_i$ of $C_{\mathscr T}$ to $\mathrm{span}\,\{e_{v_{i,n}} : n \Ge 0\}$ extends as a bounded linear operator on 
$\mathcal H_i$ and this extension is essentially equivalent to $C_{0}.$ To see that, fix $i=1, \ldots, d,$ and consider the unitary transformation $U_i : \mathcal H_i \rar \ell^2(\mathbb N)$ mapping $e_{v_{i,n}}$ to $n$th basis vector $e_n$ in the standard basis of $\ell^2(\mathbb N),$ $n \in \mathbb N.$ Note that by \eqref{action-C}, for $n \in \mathbb N,$
\beqn
U_iB_ie_{v_{i,n}}&=&
\sum_{j=0}^{\infty}\frac{1}{\mathrm{dep}(v_{i,n})+j+1}\sum_{v \in \mathrm{Chi}^{\langle j \rangle}(v_{i,n})}U_ie_v \\
&=& \sum_{j=0}^{\infty}\frac{1}{k_{\mathscr T}+n+j+1}e_{n+j}.
\eeqn 
Moreover, by \eqref{C0}, $C_0e_n = \sum_{j=0}^{\infty}\frac{1}{n+j+1}e_{n+j},$ $n \in \mathbb N.$
It follows that 
\beqn
(C_0-U_iB_iU^*_i)e_n=\sum_{j=0}^{\infty}\frac{k_{\mathscr T}}{(k_{\mathscr T}+n+j+1)(n+j+1)}e_{n+j}.
\eeqn
Thus $C_0-U_iB_iU^*_i$ has the matrix representation ${\mathfrak A}:=({\mathfrak a}_{m, n})_{m, n \Ge 0}$ with 
\beqn
{\mathfrak a}_{m, n} = \begin{cases}\frac{k_{\mathscr T}}{(k_{\mathscr T}+m+1)(m+1)}, & \mbox{if}~m \Ge n, \\
0 & \mbox{otherwise}.
\end{cases}
\eeqn
It is easy to see that $\delta(m)=\sum_{n=0}^{\infty}{\mathfrak a}_{m, n},$ $m \in \mathbb N,$ is a bounded sequence and $\gamma(n)=\sum_{m=0}^{\infty}{\mathfrak a}_{m, n},$ $n \in \mathbb N,$ is a null sequence. It now follows from \cite[Examples~V.17.4]{L1996} that  
${\mathfrak A}$ is a compact operator. This also shows that $B_i$ extends to the bounded linear operator $U^*_i(C_0-{\mathfrak A})U_i$. 
Moreover, since $\mathcal M$ is finite dimensional, $T$ and $A_i,$ $i=1, \ldots, d,$ are finite rank (bounded) linear operators. This gives the decomposition \eqref{EQ7} of $C_\mathscr{T}.$
%also shows that $C_\mathscr T$ extends as a bounded linear operator on $\ell^2(V).$ 
This completes the proof. 
\end{proof}

As an application of Lemma~\ref{point-sp-*}, we see that the closed unit disc centered at $1$ is contained in the spectrum of the Ces\'aro operator. 
It turns out that actually equality holds. 
\begin{Lemma} \label{spectrum}
Let $\mathscr{T} = (V, E)$ be a leafless, locally finite rooted directed tree of finite branching index. Then $\sigma(C_\mathscr{T})=\overline{\mathbb D_1(1)}.$ 
\end{Lemma}
\begin{proof} By assumption, $d := \mathrm{card}(\mathrm{Chi}^{\langle k_{\mathscr{T}} \rangle}(\mathsf{root})) < \infty.$ Hence, 
by Lemma~\ref{lem-decom}, 
$C_\mathscr T$ is essentially equivalent to $d$-fold direct sum $C^{(d)}_0$ of $C_0.$ Since Weyl spectrum is invariant under compact perturbation, we have, 
\beq \label{Weyl-C-T}
\sigma_w(C_\mathscr T)=  \sigma_w(C^{(d)}_0).
\eeq
Since $C_0$ is hyponormal and $\sigma_p(C_0)=\emptyset$ (see parts (a) and (c) of Theorem~\ref{THM1}), by \cite[Corollary~5.6]{B1970} (or \cite[Proposition~II.4.11]{C1991}), 
$\sigma(C^{(d)}_0)=\sigma_w(C^{(d)}_0).$
Hence, by Theorem~\ref{THM1}(a), 
$\sigma(C^{(d)}_0)=\sigma(C_0)=\overline{\mathbb D_1(1)}.$
This combined with \eqref{Weyl-C-T} shows that 
$\sigma_w(C_\mathscr T)=\overline{\mathbb D_1(1)}.$
However, since $\sigma_p(C_\mathscr T)=\emptyset$ (see Lemma~\ref{point-sp}), by \cite[Proposition 2.10]{B1970},
$\sigma(C_\mathscr T)=\sigma_w(C_\mathscr T)=\overline{\mathbb D_1(1)}.$ This yields the first part. 
\end{proof} 
The proof above shows that 
$\sigma_w(C_\mathscr T)=\overline{\mathbb D_1(1)}.$ The same conclusion holds for the essential spectrum. 
It turns out except the classical case, there are no subnormal Ces\'aro operators. 
\begin{Lemma} \label{subnormality}
Let $\mathscr{T} = (V, E)$ be a locally finite rooted directed tree of finite branching index. 
If the Ces\'aro operator $C_\mathscr{T}$ is hyponormal, then 
$\mathscr{T}$ is isomorphic to the directed tree $\mathscr T_1$. 
%following are equivalent:
%\begin{enumerate}
%\item[(i)] The Ces\'aro operator $C_\mathscr{T}$ is subnormal;
%\item[(ii)] The Ces\'aro operator $C_\mathscr{T}$ is hyponormal;
%\item[(iii)] $\mathscr{T}$ is isomorphic to $\mathscr T_1$.
%\end{enumerate}
%If, in addition, $\mathscr T$ is locally finite, then 
%Moreover, the Ces\'aro operator $C_\mathscr{T}$ is essentially equivalent to a subnormal operator.  
\end{Lemma}
\begin{proof}
%Since every subnormal operator is hyponormal (\cite[Proposition~II.4.2]{C1991}),  (i) $\Rightarrow$ (ii). Also, (iii) $\Rightarrow$ (i) follows from Theorem~\ref{THM1}(c) and Example~\ref{classical-example}. To see the remaining implication, 
Suppose that $\mathscr{T}$ is not isomorphic to $\mathbb{N}.$ This together with \eqref{branch-index} implies that there exists a vertex $u \in V_{\prec}$ such that $\text{dep}(u)+1=k_{\mathscr T}$.  
For $i=1, 2,$ choose distinct vertices $v_i \in \text{Chi}(u)$ such that $\text{dep}(v_i) = k_{\mathscr{T}}.$ It follows that $\text{Chi}^{\langle j \rangle}(v_i)=\{v_{i,j}\}$ for some vertices $v_{i,j} \in V,$ $i=1, 2,$ $j \Ge 0.$ 
By \eqref{action-C} and \eqref{action-C*}, 
\beqn
C_\mathscr{T}e_{v_i} = \sum_{j=0}^{\infty}\frac{1}{k_{\mathscr{T}}+j+1}e_{v_{i, j}}, ~ 
C^*_\mathscr{T}e_{v_i} = \frac{1}{k_{\mathscr{T}}+1}\sum_{j=0}^{k_{\mathscr{T}}}
e_{\mathrm{par}^j(e_{v_i})}, \quad i=1, 2.
\eeqn
It follows that for $f=e_{v_1}+e_{v_2},$ 
\beqn
\|C_\mathscr{T}f\|^2=2\sum_{j=0}^{\infty}\frac{1}{(k_{\mathscr{T}}+j+1)^2}, \quad \|C^*_\mathscr{T}f\|^2=
\frac{2}{(k_{\mathscr{T}}+1)^2}+ 
\frac{4k_{\mathscr{T}}}{(k_{\mathscr{T}}+1)^2}.
\eeqn
Consequently, \beq 
\label{commutator-exp}
\|C_\mathscr{T}f\|^2 - \|C^*_\mathscr{T}f\|^2=2\sum_{j=1}^{\infty}\frac{1}{(k_{\mathscr{T}}+j+1)^2} - 
\frac{4k_{\mathscr{T}}}{(k_{\mathscr{T}}+1)^2}.
\eeq
Now, we consider two cases. If $k_{\mathscr{T}} = 1,$ then we have
\begin{align*}
\|C f \|^2 - \|C^{\ast} f \|^2  & =2\sum_{j=1}^{\infty}\frac{1}{(j+2)^2} - 
1 = \frac{\pi^2}{3} - \frac{7}{2} \approx -0.2101,
\end{align*}
which clearly shows that $C_{\mathscr T}$ is not hyponormal in this case. We may now assume that $k_{\mathscr{T}} > 1.$ By the integral test (see \cite[Theorem~9.2.6]{B2000}),
\beqn \sum_{j=k_{\mathscr{T}}+2}^{\infty}\frac{1}{j^2} \Le \int_{k_{\mathscr{T}}+1}^{\infty}\frac{1}{x^2}dx = \frac{1}{k_{\mathscr{T}}+1},
\eeqn
and hence by \eqref{commutator-exp},
\beqn
\|C_\mathscr{T}f\|^2 - \|C^*_\mathscr{T}f\|^2 \Le \frac{2}{k_{\mathscr{T}}+1}\Big(1 - 
\frac{2k_{\mathscr{T}}}{k_{\mathscr{T}}+1}\Big)=2\Bigg(\frac{1 - k_{\mathscr{T}}}{(k_{\mathscr{T}} +1)^2}\Bigg) < 0.
\eeqn
This shows that $C_{\mathscr T}$ is not hyponormal in this case also.  
%This completes the proof. 
%of the equivalence of (i)-(iii). 
\end{proof}

\begin{proof}[Proof of Theorem~\ref{main-thm}]
The boundedness of $C_{\mathscr T}$ and the etsimates \eqref{norm-estimate} follow from Lemmata~\ref{bounded-narrow}$\&$\ref{point-sp-*}. The part (a) follows from Lemma~\ref{point-sp}. The part (b) is precisely stated in Lemma~\ref{point-sp-*}. 
To see the remaining parts, note that if $\mathscr T$ is a leafless narrow rooted directed tree, then by Remark~\ref{rmk-narrow-branching}(2), $\mathscr T$ is of finite branching index. Parts (c)-(e) are now consequences of Lemmata~\ref{spectrum}$\&$\ref{subnormality} and Theorem~\ref{THM1}(c).  
\end{proof}

We conclude this section with a generalization of \cite[Main Theorem]{G2023}.
\begin{Corollary} 
Let $\mathscr{T} = (V, E)$ be a locally finite rooted directed tree of finite branching index. Then $C_\mathscr{T}$ is essentially normal.
\end{Corollary}
\begin{proof} 
By Lemma~\ref{lem-decom}, 
\beq
\label{d-fold-ess-equi}
\mbox{$C_\mathscr{T}$ is essentially equivalent to the $d$-fold direct sum $C^{(d)}_0$ of $C_0,$} 
\eeq
where $d := \mathrm{card}(\mathrm{Chi}^{\langle k_{\mathscr{T}} \rangle}(\mathsf{root})) \in \mathbb N.$ 
%Since $C_0$ is subnormal (see Theorem~\ref{THM1}), 
%this completes the proof.
Since $C_0$ is a cyclic subnormal operator (see Theorem~\ref{THM1}(c)), by the Berger-Shaw theorem (see \cite[Theorem~IV.2.1]{C1991}), $[C^*_0, C_0]$ is a trace-class operator. 
One may now apply 
\eqref{d-fold-ess-equi} and use the fact that a trace-class operator is compact.  
\end{proof}

\section{Two examples and concluding remarks}

In this short section, 
we show with the help of an example that there exists a locally finite rooted directed tree $\mathscr T$ of finite branching index such that the Ces\'aro operator $C_\mathscr T$ has norm as big as one may wish. 
\begin{example} \label{k-child-ex}
For a positive integer $k \Ge 2,$ consider now the directed tree $\mathscr T_k$ with set of vertices 
$$V:=\{(0,0)\}\cup\{(1,j), \ldots, (k,j), : j \Ge 1\}$$ and
$\mathsf{root}=(0,0)$. We further require that $\mathsf{Chi}(0,0)=\{(1,1),\ldots, (k,1)\}$ and 
\[\mathsf{Chi}(i,j)=\{(i,j+1)\}, \quad i = 1, \ldots, k, ~j \Ge 1. \]
Since $\mathrm{dep}(i, j)=j$ and $\text{par}^l(i, j)=(i, j-l)$ for $i=1, \ldots, k,$ $j \Ge 1$ and $0 \Le l \Le j,$ 
it follows that for any $f \in \mathcal D(C_{\mathscr{T}_k}),$ 
\beqn
(C_{\mathscr{T}_k}f)(0, 0) &=& f(0, 0), \\ (C_{\mathscr{T}_k}f)(i, j) &=&  
\frac{1}{j+1}\Big(f(0, 0)+ \sum\limits_{k=1}^{j} f(i, k)\Big), ~ i=1, \ldots, k, ~j \Ge 1.
\eeqn  
It is worth comparing the matrix representations of $C_{\mathscr T_1}$ (see Example~\ref{classical-example}) with that of $C_{\mathscr T_2}$ (with respect to the ordered basis $\{e_{(0, 0)}, e_{(1, 1)}, e_{(2, 1)}, e_{(1, 2)}, e_{(2, 2)}, \ldots \}$):
$$C_{\mathscr T_1} = 
	\begin{bmatrix}
		1 & 0 & 0 & 0 & 0 & \cdots \\
		\frac{1}{2} & \frac{1}{2} & 0 & 0 & 0 & \cdots\\
		\frac{1}{3} & \frac{1}{3} & \frac{1}{3} & 0 & 0 & \cdots\\
		\frac{1}{4} & \frac{1}{4} & \frac{1}{4} & \frac{1}{4} & 0 &\cdots\\
		\frac{1}{5} & \frac{1}{5} & \frac{1}{5} &\frac{1}{5} & \frac{1}{5} & \cdots\\
		\vdots & \vdots & \vdots & \vdots & \vdots &\ddots
	\end{bmatrix}, \quad C_{\mathscr T_2} = 
	\begin{bmatrix}
		1 & 0 & 0 & 0 & 0 & \cdots \\
		\frac{1}{2} & \frac{1}{2} & 0 & 0 & 0 & \cdots\\
		\frac{1}{2} & 0 & \frac{1}{2} & 0 & 0 &  \cdots\\
		\frac{1}{3} & \frac{1}{3} & 0 & \frac{1}{3} & 0 & \cdots\\
		\frac{1}{3} & 0 & \frac{1}{3} & 0 &\frac{1}{3} & \cdots\\
%		\frac{1}{4} & \frac{1}{4} & 0 &\frac{1}{4} & 0 & \frac{1}{4} & 0 & \cdots\\
		\vdots & \vdots & \vdots & \vdots & \vdots &\ddots
	\end{bmatrix}.$$
An inspection of these matrices suggests that although both these matrices have the same set of entries, in the latter case, each branch of the directed tree replicates a row from $C_{\mathscr T_1}$ more than once with ``additional" zeros with a particular pattern (cf. \cite[Eq.\,(2.5)]{MR2023}).  

To estimate the norm of $C_{\mathscr T_k},$ note that by \eqref{action-C},
\beqn
C_{\mathscr T_k}(e_{(0, 0)}) = e_{(0, 0)} + \sum_{j=1}^{\infty}\frac{1}{j+1}\sum_{i=1}^k e_{(i, j)},
%1 + (\frac{1}{2^2} + \frac{1}{2^2}) + \frac{1}{2^2} + (\frac{1}{3^2} + \frac{1}{3^2}) + \frac{1}{3^2} + \cdots  = \frac{\pi^2}{6}+2(\frac{\pi^2}{6}-1),  
\eeqn
and hence 
\beqn
\|C_{\mathscr T_k}(e_{(0, 0)}\|^2  = 1 + k\big(\frac{\pi^2}{6}-1\big).
\eeqn
In particular, unlike the case of $C_{\mathscr T_1}$ (see Theorem~\ref{THM1}), we have $\|C_{\mathscr T_k}\| > 2$ provided $k \Ge 5.$ Since the spectral radius of $C_{\mathscr T_k}$ is always $2$ (see Lemma~\ref{spectrum}), $C_{\mathscr T_k}$ is not normaloid for $k \Ge 5$. 
\eof
\end{example}

A careful examination of Example~\ref{k-child-ex} shows that
for any locally finite rooted directed tree $\mathscr T$ of finite branching index, the norm of the Ces\'aro operator $C_\mathscr T$ satisfies
\beqn
\|C_{\mathscr T}\| \Ge \sup_{w \in V_{\prec}} \text{card}(\text{Chi}(w)) \Big(\sum_{j=\text{dep}(w)+1}^{\infty} \frac{1}{j^2}\Big)^{\frac{1}{2}}.
\eeqn
%where $\deg\,w$ denotes the cardinality of $\text{Chi}(w).$
It would be interesting to find a handy formula for the norm of $C_{\mathscr T}$ for a locally finite rooted directed tree $\mathscr T$ of finite branching index. When $C_{\mathscr T}$ is normaloid? Another unanswered question is whether the point spectrum of $C^*_{\mathscr T}$ contains any point from the complement of $\mathbb D_1(1)$?

The following example shows that the assumption that $\mathscr T$ is a narrow directed tree can not be relaxed.
\begin{figure}[H]
\begin{center}
\begin{tikzpicture}[
dot/.style = {circle, fill=black, inner sep=1.8pt},
>=Stealth, % arrow tip
node distance=1cm
]
% coordinates (adjust numbers if you want a different look)
\coordinate (v0)  at (-4,-2);
\coordinate (v1)  at (-2,-2);
\coordinate (v2)  at (0,-2);
\coordinate (v3)  at (2,-2);
		
		\coordinate (w11) at (-2,-4);
		\coordinate (w21) at (0,-4);
		\coordinate (w22) at (0,0);
		\coordinate (w31) at (2,-4);
		\coordinate (w32) at (4,0);
		\coordinate (w33) at (2,0);
		
		% edges (directed)
		\draw[->, thick] (v0) -- (v1);
		\draw[->, thick] (v1) -- (v2);
		\draw[->, thick] (v2) -- (v3);
		
		\draw[->, thick] (v1) -- (w11);
		\draw[->, thick] (v2) -- (w21);
		\draw[->, thick] (v2) -- (w22);
		\draw[->, thick] (v3) -- (w31);
		\draw[->, thick] (v3) -- (w32);
		\draw[->, thick] (v3) -- (w33);
		
		\draw[->, thick, dashed] (v3) -- ++(2,0);
		
		% nodes (dots)
		\foreach \p in {v0,v1,v2,v3,w11,w21,w22,w31,w32,w33}
		\node[dot] at (\p) {};
		
		% labels (tweak positioning as you like)
		\node[below right]       at (v0) {$v_0$};
		\node[below right]       at (v1) {$v_1$};
		\node[below right] at (v2) {$v_2$};
		\node[below right]       at (v3) {$v_3$};
		
		\node[left] at (w11) {$w_{11}$};
		\node[left] at (w21) {$w_{21}$};
		\node[left]  at (w22) {$w_{22}$};
		
		\node[left] at (w31) {$w_{31}$};
		\node[left] at (w32) {$w_{32}$};
		\node[left] at (w33) {$w_{33}$};
		
	\end{tikzpicture}
	\caption{A rooted directed tree which is not narrow} 
\label{fig1}
\end{center}
\end{figure}

\begin{example} \label{Chi-n-order-1}
Let $V = \{v_n : n \in \mathbb Z_+\} \cup \{w_{n,j}:n\in \mathbb N \text{ and } 1 \Le j \Le n\}.$ Consider the rooted directed tree $\mathscr T=(V, E),$ where $E$ is governed by
\beqn
\text{Chi}(v_0)=\{v_{1}\}, \quad
\text{Chi}(v_n)=\{v_{n+1}\} \cup \{w_{n,j}:n\in \mathbb N \text{ and } 1 \Le j \Le n\}, \, j \in \mathbb N
\eeqn
(see Figure~\ref{fig1}).
For $f \in \ell^2(V)$ and $n \in \mathbb Z_+$, let $K_n = \sum\limits_{j = 0}^n f(v_j).$ 
Observe that 
%\begin{align*}
%(C_{\mathscr{T}} f)(v_n)  = \frac{1}{\operatorname{depth} (v_n)+1} \sum\limits_{j=0}^{\operatorname{depth} (v_n)} f(\operatorname{par}^j (v_n))
% = \frac{1}{n+1} \sum\limits_{j=0}^{n} f(v_j)
% = \frac{K_n}{n+1}
%\end{align*} and 
for $1 \Le j \Le n,$ 
\begin{align*}
(C_{\mathscr{T}} f) (w_{n,j}) & = \frac{1}{\operatorname{depth} (w_{n,j})+1} \sum\limits_{j=0}^{\operatorname{depth} (w_{n,j})} f(\operatorname{par}^j (w_{n,j}))\\
& = \frac{1}{n+2} \Big(f(w_{n,j}) + \sum\limits_{j=0}^{n} f(v_j)\Big)\\
& = \frac{f(w_{n,j}) + K_n}{n+2}.
\end{align*}
Therefore, we see that
\beq \label{norm-exp-ex2}
\|C_{\mathscr{T}} f\|^2 
%& = & \sum\limits_{n=0}^{\infty} |C_{\mathscr{T}} f (v_n)|^2 + \sum\limits_{n=1}^{\infty} \sum\limits_{j=1}^{n} |C_{\mathscr{T}} f(w_{n,j})|^2 \notag\\
& \Ge & 
%\sum\limits_{n=0}^{\infty} \Bigg|\frac{K_n}{n+1}\Bigg|^2 + 
\sum\limits_{n=1}^{\infty} \sum\limits_{j=1}^{n} \Bigg|\frac{f(w_{n,j}) + K_n}{n+2}\Bigg|^2.
\eeq
Consider the function $f \in \ell^2(V)$ given by
\beqn
f(v)=\begin{cases} \frac{1}{n} & \mbox{if~}v=v_n, \\
0 & \mbox{otherwise}.
\end{cases}
\eeqn
It now follows from \eqref{norm-exp-ex2} that
\beqn
\|C_{\mathscr{T}} f\|^2 \Ge \sum\limits_{n=1}^{\infty} \sum\limits_{j=1}^{n} \Bigg|\frac{K_n}{n+2}\Bigg|^2 = 
\sum\limits_{n=1}^{\infty} n\Bigg|\frac{K_n}{n+2}\Bigg|^2,
\eeqn
which is divergent.
This shows that $C_{\mathscr{T}}$ is not a bounded linear operator on $\mathscr T.$ 
\eof
\end{example}

%\section{Concluding remarks}

Let $\mathscr{T} = (V, E)$ be a directed tree. For $\alpha \in [0, \infty),$ set
\beqn M_{\alpha, j} := \frac{\mathrm{card}(\mathrm{Chi}^{\langle j \rangle}(\mathsf{root}))}{(j+1)^{\alpha}}, \quad j \Ge 0,
\eeqn
and note the following:
\begin{enumerate}
\item[$\bullet$] If $C_\mathscr T$ is a bounded linear operator, then by \eqref{bdd-limit0}, $\lim_{j \rar \infty} M_{2, j}=0.$  
\item[$\bullet$] We have seen in Theorem~\ref{main-thm} that if $\sup_{j \Ge 0}M_{0, j} < \infty$, then the Ces\'aro operator $C_\mathscr T$ is bounded. 
\item[$\bullet$] In Example~\ref{Chi-n-order-1}, we construct a rooted directed tree with $\sup_{j \Ge 0}M_{1, j} < \infty$ for which $C_\mathscr T$ is unbounded. 
\end{enumerate}
This leads us to ask: under the assumption that there exists some $\alpha  \in (0, 1)$
 for which $\sup_{j \Ge 0}M_{\alpha, j} < \infty$,
is the operator $C_\mathscr T$
on the rooted directed tree 
$\mathscr T$ necessarily bounded?

\end{document}